\documentclass[3p,12pt]{elsarticle}

\usepackage{latexsym,amssymb,amsmath,amsthm,amsfonts}
\usepackage{graphicx}
\usepackage{epstopdf}
\usepackage{caption}

\theoremstyle{plain}
\newtheorem{theorem}{Theorem}
\newtheorem{lemma}{Lemma}

\newtheorem{definition}{Definition}
\theoremstyle{remark}
\newtheorem{rem}{Remark}
\newtheorem{example}{Example}
\newtheorem{cor}{Corollary}

\journal{Journal of Mathematical Analysis and Applications}

\begin{document}

\begin{frontmatter}

\title{\textbf{Whittaker-Kotel'nikov-Shannon approximation of
$\varphi$-sub-Gaussian random processes}\\[5mm]
\small \textit{Short title: WKS approximation of
$Sub_{\varphi}(\Omega)$ processes}}

\author[ku]{Yuriy Kozachenko}
\ead{ykoz@ukr.net}

\author[la]{Andriy Olenko\corref{cor1}}
\ead{a.olenko@latrobe.edu.au}

\cortext[cor1]{Corresponding author. Phone: +61-3-9479-2609 \quad  Fax:  +61-3-9479-2466}
\address[ku]{Department of Probability Theory, Statistics and Actuarial Mathematics, Kyiv University, Kyiv, Ukraine}
\address[la]{Department of Mathematics and Statistics, La Trobe University, Victoria 3086, Australia}

\begin{abstract}
The article starts with generalizations of some classical results and new truncation error upper bounds in the sampling theorem for bandlimited stochastic processes.
Then, it investigates $L_p([0,T])$ and uniform approximations of $\varphi$-sub-Gaussian random processes by finite time sampling sums. Explicit truncation error upper bounds are established. Some specifications of the general results for which the assumptions can be easily verified are given. Direct analytical methods are employed to obtain the results.
\end{abstract}

\begin{keyword}
Sampling theorem \sep Truncation error upper bound \sep Convergence in $L_p([0,T])$ \sep Uniform convergence \sep Sub-Gaussian random process \sep  Bandlimited stochastic process.

\MSC  42C15 \sep 60G15 \sep 94A20 
\end{keyword}

\end{frontmatter}

\section{Introduction}
Recovering a continuous function from discrete samples and assessing the information lost are the fundamental problems in sampling theory and signal processing. Whittaker-Kotel'nikov-Shannon (WKS) theorems allow the coding of a continuous band-limited signal by a sequence of its discrete samples without the loss of information. On the other hand sampling results are important not only because of signal processing applications. WKS theorems are equivalent to various fundamental results
in mathematics, see, e.g.,  \cite{ but, hig, sma}. Therefore, they are also valuable for theoretical studies. In spite of the substantial progress in modern approximation methods (especially wavelets) WKS type expansions are still of great importance and numerous new refine results are published regularly by engineering and mathematics researchers, see, e.g., the recent volumes  \cite{hog,moe,zay} in Birkh\"{a}user's Applied and Numerical Harmonic Analysis series.

Despite extensive investigations of sampling expansions  of  deterministic signals there has been remarkably little fundamental theoretical study for the case of stochastic signals. The publications  \cite{he, ole2, pog}, and references therein  present an almost exhaustive survey of key results and approaches in stochastic sampling theory. 

The development of stochastic sampling theory  began with the truncation error upper bounds given by  \cite{bel,cam,pir}.
Using their pioneering approaches the majority of recent stochastic sampling results were obtained for harmonizable stochastic processes.  Spectral representations of  these stochastic processes and an inner product preserving isomorphism were used to employ  deterministic sampling results and error bounds for finding mean square approximation errors  for harmonizable stochastic processes, see, e.g.,  \cite{bel, he, ole1, ole2, pir, pog} and references therein. 

However, this approach is not applicable for other classes of stochastic processes or other measures of deviation.    For example, for various practical applications one needs to require uniform convergence  instead of the mean-square one.  Also, from a practical point of view, measures of the closeness of trajectories are often more appropriate than estimates of mean-square errors in each time point. Controlling signal distortions in the mean-square sense may result in situations where relevant signal features are substantially locally distorted. Instead of small mean-square errors one may need to guarantee that the signal values have not been changed more than a certain tolerance. For example,  near-lossless compression requires small user-defined tolerance levels, see \cite{che,  har}. Also, it is often required to give an adequate description of the performance of approximations in both cases, for points where signals are relatively smooth and points where spikes occur. The uniform measure of closeness of trajectories  maintains equal precision  throughout  the  entire  signal  support. It indicates the necessity of elaborating special techniques. Recently a considerable attention was given to wavelet orthonormal series representations of stochastic processes. Some new results and references on convergence of wavelet expansions of random processes can be found in  \cite{kozol1, kozol2}. WKS sampling is an important example of such expansions and requires specific methods and techniques.

The analysis and the approach presented in the paper contribute to these investigations  in the former sampling literature. Sampling truncation errors for new classes of stochastic processes and probability metrics are given.  Novel techniques to approximate  sub-Gaussian random processes with given accuracy and reliability are developed. Finally, it should be mentioned that the analysis of the rate of convergence gives a constructive algorithm for determining the number of terms in the WKS expansions to ensure the uniform approximation of stochastic processes with given accuracy.

The article derives sampling results for two classes of the so-called $\varphi$-sub-Gaussian random processes. These classes play an important role in  extensions of various properties of Gaussian processes to more general settings. To the best of our knowledge, the WKS expansions have never been studied for sub-Gaussian random processes and using  $L_p([0,T])$ and  uniform probability metrics. This work was intended as an attempt to obtain first results in this direction.

Note that even for the case of Gaussian processes the obtained  sampling results and methodology are new. There are no known results on  $L_p([0,T])$ and uniform sampling approximations of Gaussian processes in the literature.

The article is organized as follows. First, it generalizes some Belyaev's results. Then, Section 3 introduces two classes of $\varphi$-sub-Gaussian random processes. Section 4 presents results on the approximation of $\varphi$-sub-Gaussian random processes in  $L_p([0,T])$ with a given accuracy and reliability.  Section 5 establishes explicit truncation error upper bounds in uniform sampling approximations of $\varphi$-sub-Gaussian random processes. Finally, short conclusions and some problems for further investigation are presented in Section~\ref{sec6}. 

We use direct analytical and  probability methods to obtain all results. 
Some computations and plotting in Example~\ref{ex2} were performed by using Maple 17.0 of Waterloo Maple~Inc. 

In what follows $C$ denotes constants which are not important for our exposition. Moreover, the same symbol  may be used for different constants appearing in the same proof.

\section{Kotel'nikov-Shannon stochastic sampling}

Known deterministic sampling methods often may not be appropriate to approximate stochastic processes and to estimate stochastic reconstruction errors. Since random signals  play a key role in modern signal processing new refined sampling results for stochastic processes are required. 

This section generalizes some results in  \cite{bel} and obtains new truncation error upper bounds in the WKS sampling theorem for bandlimited stochastic processes. 

 Let $\mathbf X(t),$ $t\in \mathbf{R},$ be a stationary
random process with $ \mathbf E\mathbf \mathbf X(t)=0$ whose spectrum is bandlimited to $[-\Lambda,\Lambda),$ that is
\[\mathbf B(\tau):=\mathbf E\mathbf X(t+\tau) \mathbf X(t)= \int_{-\Lambda}^\Lambda e^{i\tau \lambda} dF(\lambda),\] 
where $F(\cdot)$ is the spectral function of $\mathbf X(t).$ The  process $ \mathbf X(t)$ can be represented as 
\begin{equation}\label{spect}\mathbf X(t)= \int_{-\Lambda}^\Lambda e^{it \lambda} d\Phi(\lambda),
\end{equation} 
where $\Phi(\cdot)$ is a random measure on $\mathbb R$ such that $\mathbf E\left[ \Phi(\Delta_1)\Phi(\Delta_2)\right]=F(\Delta_1\cap \Delta_2)$ for any measurable sets $\Delta_1, \Delta_2\subset \mathbb R.$

Then, for all $\omega>\Lambda$ there holds
\begin{equation}\label{shan1}
\mathbf X(t)=\sum_{k=-\infty}^\infty \frac{\sin\left(\omega\left(t-\frac{k\pi}{\omega}\right)\right)}{\omega\left(t-\frac{k\pi}{\omega}\right)}\, \mathbf X\left(\frac{k\pi}{\omega}\right),
\end{equation}
and the series (\ref{shan1}) converges uniformly in mean square, see, for example,  \cite{bel}.

Let us consider the truncation version of (\ref{shan1}) given by the formula
\begin{equation}\label{shan2}
\mathbf X_n(t):=\sum_{k=-n}^n \frac{\sin\left(\omega\left(t-\frac{k\pi}{\omega}\right)\right)}{\omega\left(t-\frac{k\pi}{\omega}\right)}\, \mathbf X\left(\frac{k\pi}{\omega}\right).
\end{equation}

In his classical paper \cite{bel} Belyaev proved  a sampling theorem for random processes with bounded spectra. The key ingredient in obtaining the main result was an explicit upper bound of the reconstruction error. In the above notations, the bound can be written as

\[\mathbf E \left|\mathbf X(t)-\mathbf X_n(t)\right|^2\le \frac{\mathbf 16 \omega^2(2\pi+ t\omega)^2 B(0)}{\pi^4n^2\left(1-\frac{\Lambda}{\omega}\right)^2}.\]

Part 1 of Theorem~\ref{the0} below generalizes this result, while part 2 obtains novel bounds for increments of the stochastic process $\mathbf X(t)-\mathbf X_n(t).$  Note that \cite{bel} has no results for increments analogous to  those reported in part 2.

 \begin{theorem}\label{the0}  Let $z\in (0,1),$ $t >0,$ $s>0.$ Then  
\begin{enumerate}
\item[\rm 1.]  for $n \ge \frac{\omega t}{\pi \sqrt{z}}$ it holds that
\[\mathbf E \left|\mathbf X(t)-\mathbf X_n(t)\right|^2\le n^{-2}\,C_n(t),\]
where 
\begin{equation}\label{Cn}C_n(t):=\mathbf B(0)\cdot\left(\frac{4\omega t}{\pi^2 (1-z)}+\frac{4\left(z+1+\frac{1}{n}\right)}{\pi (1-z)^2\left(1-\frac{\Lambda}{\omega}\right)} \right)^2;
\end{equation}
\item[\rm 2.] for $n \ge\frac{\omega}{\pi \sqrt{z}}\max(t,s)$ it holds that
\[\mathbf E \left(\mathbf Y_n(t)-\mathbf Y_n(s)\right)^2\le \left(\frac{t-s}{n}\right)^{2} b_n(t,s),\]
where $\mathbf Y_n(t):=\mathbf X(t)-\mathbf X_n(t),$ 
\begin{eqnarray}b_n(t,s)&:=&\mathbf B(0)\cdot\left(W_n(t,s)+\frac{Q_n(t,s)}{\left(1-\frac{\Lambda}{\omega}\right)} \right)^2,\label{bn}\\
W_n(t,s)&:=& \frac{4\omega}{\pi^2(1-z)} \left(\omega s +1+\frac{\omega^2(s+t)s}{\pi^2n^2 (1-z)} \right),\nonumber\\
Q_n(t,s)&:=&\frac{2\omega}{\pi(1-z)^2}\left({z+1+n^{-1}}+\frac{2\omega(s+t)}{n \pi^2}\right).\nonumber
\end{eqnarray}
\end{enumerate}
 \end{theorem} 
\begin{rem} The parameter $z$ was introduced to provide simple expressions for the upper bounds.  To guarantee a specified reconstruction accuracy the number of terms in parts 1 and 2 of Theorem~\ref{the0} can be selected as  $n =\lceil \frac{\omega t}{\pi \sqrt{z}}\rceil $ and $n = \lceil\frac{\omega}{\pi \sqrt{z}}\max(t,s)\rceil ,$  respectively, where $\lceil x \rceil$  denotes the smallest integer not less than $x.$
\end{rem}

To prove Theorem~\ref{the0} we need two lemmata. 

 \begin{lemma}\label{lem01}If  $0 \le n < m$ and $\nu\in(0,1],$ then 
\[\left|\sum_{k=n}^m\sin(k\pi\nu)\right|\le \frac{1}{\nu}.\]
 \end{lemma}
\begin{proof} Notice that
\[\left|\sum_{k=n}^m\sin(k\pi\nu)\right|=\left|\Im \left(\sum_{k=n}^m e^{i k\pi\nu}\right)\right|\le \left|\sum_{k=n}^m e^{i k\pi\nu}\right|\]
\[=\left|\frac{ e^{i (m+1)\pi\nu}- e^{i n\pi\nu}}{ e^{i \pi\nu}-1}\right|\le \frac{2}{ \left|e^{i \pi\nu}-1\right|}=\frac{1}{\sin\left(\frac{\pi\nu}{2}\right)}.\]
The statement of the lemma follows from the inequality $\sin(x)>\frac{2}{\pi}x,$ where $0<x<\pi/2.$
\end{proof}

 \begin{lemma}\label{lem02} If $\{a_k,k\in \mathbb{N}\}$ is a sequence of real numbers, $0 \le n < m,$ and $\nu\in(0,1],$ then 
\[\left|\sum_{k=n}^m a_k\sin(k\pi\nu)\right|\le \frac{1}{\nu}\left(\sum_{k=n}^m \left|a_{k+1}-a_k\right| +|a_{m+1}|\right).\]
 \end{lemma}
\begin{proof} By the Abel transformation 
\[\sum_{k=n}^m a_k\sin(k\pi\nu)=B_ma_{m+1}-\sum_{k=n}^m B_k\left(a_{k+1}-a_k\right),\]
where $B_k:=\sum_{l=n}^k\sin(l\pi\nu).$

Now, Lemma~\ref{lem02} follows from  Lemma~\ref{lem01}.
\end{proof}
\begin{proof} To prove Theorem~{\rm\ref{the0}} we note that it follows from the spectral representation (\ref{spect})  and 

\[e^{it\lambda}=\sum_{k=-\infty}^\infty e^{ik\pi\lambda/\omega}\,\frac{\sin\left(\omega\left(t-\frac{k\pi}{\omega}\right)\right)}{\omega\left(t-\frac{k\pi}{\omega}\right)}\]
that 
\[\mathbf X(t)-\mathbf X_n(t)=\int_{-\omega}^\omega\sum_{|k|> n} R_k(t,\lambda)d\Phi(\lambda),\]
where 
\[
R_k(t,\lambda):= e^{ik\pi\lambda/\omega}\,\frac{\sin\left(\omega\left(t-\frac{k\pi}{\omega}\right)\right)}{\omega\left(t-\frac{k\pi}{\omega}\right)}+e^{-ik\pi\lambda/\omega}\,\frac{\sin\left(\omega\left(t+\frac{k\pi}{\omega}\right)\right)}{\omega\left(t+\frac{k\pi}{\omega}\right)}\]
\begin{equation}\label{Rk}
=\frac{\sin(\omega t)}{(\omega t)^2-(k\pi)^2}\left[2\omega t\cos\left(k\pi\left(1-\frac{\lambda}{\omega}\right)\right)-2ik\pi \sin\left(k\pi\left(1-\frac{\lambda}{\omega}\right)\right)\right]
\end{equation}
and
\begin{equation}\label{est0}\mathbf E \left|\mathbf X(t)-\mathbf X_n(t)\right|^2=\int_{-\omega}^\omega\left(\sum_{|k|> n} R_k(t,\lambda)\right)^2 dF(\lambda).
\end{equation}
Let $\lambda>0$ and  $(\omega t)^2 \le z(n\pi)^2,$ $z \in (0,1).$ Notice, that by (\ref{Rk}) we obtain
\[\Im\left(\sum_{|k|> n} R_k(t,\lambda)\right)=-\sum_{|k|> n} \frac{ 2k\pi\sin(\omega t)}{(\omega t)^2-(k\pi)^2} \sin\left(k\pi\left(1-\frac{\lambda}{\omega}\right)\right).\]
Let $a_k:=\frac{ 2k\pi\sin(\omega t)}{(\omega t)^2-(k\pi)^2}.$  As $a_k\to 0$ when $k\to\infty,$ then it follows from Lemma~\ref{lem02} that
\begin{equation}\label{est1}\left|\Im\left(\sum_{|k|> n} R_k(t,\lambda)\right)\right| \le \frac{1}{1-\frac{\lambda}{\omega}}\cdot\sum_{|k|> n}|a_{k+1}-a_k|.\end{equation}
It follows from $(\omega t)^2\le z(n\pi)^2,$ $z \in (0,1),$ that for $k>n:$
\[|a_{k+1}-a_k|\le 2\pi \left|\frac{ k}{(\omega t)^2-(k\pi)^2}-\frac{ k+1}{(\omega t)^2-((k+1)\pi)^2}\right|\]
\[=  \frac{ 2\pi\left((\omega t)^2+k(k+1)\pi^2\right)}{\left( (k\pi)^2-(\omega t)^2\right)\left(((k+1)\pi)^2- (\omega t)^2\right)}\le \frac{ 2\pi^3\left(zk^2+k(k+1)\right)}{(k\pi)^4(1-z)^2} \]
\begin{equation}\label{estz}\le  \frac{2\left(z+1+n^{-1}\right)}{\pi k^2(1-z)^2}.\end{equation}
Analogously one can obtain that  (\ref{estz}) also holds for $k<-n.$

By (\ref{est1}) and (\ref{estz}) we get
\[\left|\Im\left(\sum_{|k|> n} R_k(t,\lambda)\right)\right| \le \frac{1}{1-\frac{\lambda}{\omega}}\cdot \frac{4\left(z+1+n^{-1}\right)}{\pi (1-z)^2}\sum_{k=n+1}^{+\infty}\frac{1}{k^2} \]
\begin{equation}\label{est2}\le \frac{1}{1-\frac{\Lambda}{\omega}}\cdot \frac{4\left(z+1+n^{-1}\right)}{\pi (1-z)^2}\cdot \frac{1}{n}.
\end{equation}
It follows from (\ref{Rk}) that
\[\left|\Re \left(\sum_{|k|> n} R_k(t,\lambda)\right)\right|=\left|\sum_{|k|> n} \frac{ 2\omega t\sin(\omega t)}{(\omega t)^2-(k\pi)^2} \cos\left(k\pi\left(1-\frac{\lambda}{\omega}\right)\right)\right|\]
\begin{equation}\label{est3}\le \sum_{|k|> n} \frac{ 2\omega t}{(k\pi)^2-(\omega t)^2} \le  \frac{4\omega t}{\pi^2(1-z)}\sum_{k = n+1}^{+\infty}\frac{1}{k^2}\le \frac{4\omega t}{\pi^2(1-z)}\cdot\frac{1}{n}.\end{equation}
Combining (\ref{est2}) and (\ref{est3}) we obtain
\[\left|\sum_{|k|> n} R_k(t,\lambda)\right|\le \frac{S_n(t)}{n},\]
where
\[S_n(t):= \frac{4\omega t}{\pi^2 (1-z)}+\frac{4\left(z+1+\frac{1}{n}\right)}{\pi (1-z)^2\left(1-\frac{\Lambda}{\omega}\right)}.\]
For the case $\lambda<0$ the proof is analogous.

Finally, item 1 of the theorem follows from (\ref{est0}) and the estimate
\begin{equation}\label{Dif}
\mathbf E \left|\mathbf X(t)-\mathbf X_n(t)\right|^2\le \int_{\omega}^\omega \frac{S_n^2(t)}{ n^{2}}dF(\lambda)=\frac{S_n^2(t)}{ n^{2}}\mathbf B(0)=\frac{C_n(t)}{ n^{2}}.
\end{equation} 

Now we prove item 2 of the theorem. Similarly to (\ref{est0}) it holds true that 
\[
\mathbf Y_n(t)-\mathbf Y_n(s)=\int_{-\omega}^\omega\sum_{|k|> n} \left(R_k(t,\lambda)- R_k(s,\lambda)\right) d\Phi(\lambda),
\]
\begin{equation}\label{Rk1}
\mathbf E \left|\mathbf Y_n(t)-\mathbf Y_n(s)\right|^2=\int_{-\omega}^\omega\left(\sum_{|k|> n} \left(R_k(t,\lambda)- R_k(s,\lambda)\right)\right)^2 dF(\lambda).
\end{equation}
Let $\lambda>0$ and $(\omega \max(t,s))^2 \le z(n\pi)^2,$ $z \in (0,1).$

It follows from Lemma~\ref{lem02} that
\begin{equation}\label{Dk}\left|\Im\left(\sum_{|k|> n} \left( R_k(t,\lambda) - R_k(s,\lambda)\right)\right)\right| \le \frac{1}{1-\frac{\lambda}{\omega}}\cdot\sum_{|k|> n}D_k,
\end{equation}
where
\[D_k:= 2\pi \left|\frac{k\sin(\omega t)}{(\omega t)^2-(k\pi)^2}-\frac{k\sin(\omega s)}{(\omega s)^2-(k\pi)^2}-\frac{(k+1)\sin(\omega t)}{(\omega t)^2-((k+1)\pi)^2}+\frac{(k+1)\sin(\omega s)}{(\omega s)^2-((k+1)\pi)^2}\right|. \]

We can estimate $D_k$ as follows
\[D_k \le 2\pi \left(\left|\sin(\omega t)-\sin(\omega s)\right|\left|\frac{k}{(\omega t)^2-(k\pi)^2}-\frac{k+1}{(\omega t)^2-((k+1)\pi)^2}\right|+|\sin(\omega s)|\right.\]
\[\left.\times\left|\frac{k}{(\omega t)^2-(k\pi)^2}-\frac{k+1}{(\omega t)^2-((k+1)\pi)^2}-\frac{k}{(\omega s)^2-(k\pi)^2}+\frac{k+1}{(\omega s)^2-((k+1)\pi)^2}\right|\right). \]

By the estimate
\[\left|\frac{(\omega s)^2-(\omega t)^2}{\left((\omega t)^2-(k\pi)^2\right)\left((\omega s)^2-(k\pi)^2\right)}\right|\le 
\frac{\omega^2(s+t)|t-s|}{(k\pi)^4(1-z)^2}\]
we obtain
\[\left|\frac{k}{(\omega t)^2-(k\pi)^2}-\frac{k+1}{(\omega t)^2-((k+1)\pi)^2}-\frac{k}{(\omega s)^2-(k\pi)^2}+\frac{k+1}{(\omega s)^2-((k+1)\pi)^2}\right|\]
\[\le (k+1)\left|\frac{(\omega s)^2-(\omega t)^2}{\left((\omega t)^2-((k+1)\pi)^2\right)\left((\omega s)^2-((k+1)\pi)^2\right)}\right|\]
\begin{equation}\label{est4} + k\left|\frac{(\omega s)^2-(\omega t)^2}{\left((\omega t)^2-(k\pi)^2\right)\left((\omega s)^2-(k\pi)^2\right)}\right| \le \frac{2\,\omega^2(s+t)|t-s|}{k^3 \pi^4(1-z)^2}. \end{equation}

Therefore, by (\ref{estz}), (\ref{est4}), and the inequality 
\begin{equation}\label{sin}\left|\sin(\omega t)-\sin(\omega s)\right|\le 2\left|\sin\left(\frac{\omega( t-s)}{2}\right)\right|\le |t-s|\cdot\omega
\end{equation}
we get
\[D_k\le {}\left(\frac{2\omega|t-s|\left(z+1+n^{-1}\right)}{\pi k^2(1-z)^2}+\frac{4\omega^2(s+t)|t-s|}{ \pi^3 k^3(1-z)^2}\right)\]
\[=|t-s|\cdot \frac{2\omega}{\pi(1-z)^2}\left({z+1+n^{-1}}+\frac{2\omega(s+t)}{k \pi^2}\right)\cdot \frac{1}{k^2}.\]
Hence, it follows from (\ref{Dk}) that
\begin{equation}\label{Dk1}\left|\Im\left(\sum_{|k|> n} \left( R_k(t,\lambda) - R_k(s,\lambda)\right)\right)\right| \le \frac{1}{1-\frac{\Lambda}{\omega}}\cdot|t-s|\cdot\frac{Q_n(t,s)}{n}.
\end{equation}

Notice that
\[\left|\Re\left(\sum_{|k|> n} \left( R_k(t,\lambda) - R_k(s,\lambda)\right)\right)\right| = \left|\sum_{|k|> n} \left(\frac{ 2\omega t\sin(\omega t)}{(\omega t)^2-(k\pi)^2}-\frac{ 2\omega s\sin(\omega s)}{(\omega s)^2-(k\pi)^2}\right)\right.\]
\[\left.\times \cos\left(k\pi\left(1-\frac{\lambda}{\omega}\right)\right)\right| \le 4\omega \sum_{k=n+1}^{+\infty} \Delta_k,
\]
where 
\[\Delta_k:= \left|\frac{ t\sin(\omega t)}{(\omega t)^2-(k\pi)^2}-\frac{ s\sin(\omega s)}{(\omega s)^2-(k\pi)^2}\right|.\]
By (\ref{sin}) we estimate $\Delta_k$ as follows
\begin{eqnarray}\Delta_k &\le& \left|\frac{ t\sin(\omega t)-s\sin(\omega s)}{(\omega t)^2-(k\pi)^2}\right|+ s\left|\frac{ 1}{(\omega t)^2-(k\pi)^2} - \frac{ 1}{(\omega s)^2-(k\pi)^2}  \right|\nonumber\\
&\le&  \frac{\omega^2s(s+t)|t-s|}{k^4 \pi^4(1-z)^2} + \frac{ |t-s|\cdot|\sin(\omega t)|+s|\sin(\omega t)-\sin(\omega s)|}{(1-z)(k\pi)^2}\\
&\le& 
\frac{\omega^2s(s+t)|t-s|}{k^4 \pi^4(1-z)^2}\nonumber+ |t-s|\,\frac{\omega s+1}{(1-z)(k\pi)^2} \le \frac{|t-s|}{4\omega k^2}\,W_n(t,s).\nonumber
\end{eqnarray}
Hence, we get
\begin{equation}\label{Dk2}\left|\Re\left(\sum_{|k|> n} \left( R_k(t,\lambda) - R_k(s,\lambda)\right)\right)\right| \le \frac{|t-s|}{n}\,W_n(t,s).
\end{equation}
Combining (\ref{Dk1}) and (\ref{Dk2}) we obtain
\[\left|\sum_{|k|> n} \left( R_k(t,\lambda) - R_k(s,\lambda)\right)\right|\le \frac{|t-s|}{n}\left(W_n(t,s)+\frac{Q_n(t,s)}{\left(1-\frac{\Lambda}{\omega}\right)}\right).\]

For the case $\lambda<0$ the proof is similar.

Finally,  analogously to  the derivations in (\ref{Dif}), one can deduce statement 2 of the theorem from (\ref{Rk1}). 
\end{proof}
\section{$\varphi$-sub-Gaussian random processes}
In their pioneering papers \cite{bel,pir} Belyaev and Piranashvili extended the deterministic sampling theory to classes of analytic stochastic processes. Almost all trajectories of these processes can be analytically continued. Recently, there have been  considerable efforts to develop the WKS sampling theory to new classes of stochastic processes. 

This section reviews the definition of $\varphi$-sub-Gaussian random processes and  their relevant properties. 

Tail distributions of sub-Gaussian random variables behave similarly to the Gaussian ones so that sample path properties of sub-Gaussian processes rely on their mean square regularity. One of the main classical tools  to study the boundedness
of sub-Gaussian processes was metric entropy integral estimates by Dudley \cite{dud0}.  These results were extended by Fernique \cite{fern} and Ledoux and Talagrand \cite{led} using the generic chaining (majorizing measures) method.  There is a rich and well-developed theory on bounding  sub-Gaussian random variables and processes, therefore below we cite only some key publications related to our approach.  Good introductions on bounding stochastic processes can be found in the classical monographs \cite{dud,  led1, led, tal1, tal2} and references therein. Regularity estimates under non-Gaussian assumptions were derived in \cite{csa}.  A novel approach based on  Malliavin derivatives was proposed in~\cite{vie}.

Some of these results can also be used to obtain bounds for Gaussian or sub-Gaussian random processes which are similar to the ones derived in this article. However, we employ specific results and methods for the $\varphi$-sub-Gaussian case. These methods are often simpler than the generic chaining or  Malliavin-derivative-based concentration results. Moreover, they are in ready-to-use forms for the considered sampling problems.

The space of $\varphi$-sub-Gaussian random variables  was introduced in the paper~\cite{koz4} to generalize the class of sub-Gaussian random variables defined in \cite{kah}.  Various properties of the space of $\varphi$-sub-Gaussian random variables were studied in the book  \cite{bul} and the article  \cite{ant}. More information on sub-Gaussian and $\varphi$-sub-Gaussian random processes and their applications can be found in the publications  \cite{bie, bul, ant1,  fer, kozvas,  yam}.

 \begin{definition} {\rm  \cite{kra}} A continuous even convex function $\varphi(x),$ $x\in {\mathbb R},$ is called
an Orlicz N-function, if it is monotonically increasing for $x>0$,
$\varphi(0)=0,$ $ {{\varphi(x)}/{x}}\to 0,$ when $x\to 0,$ and $ {{\varphi(x)}/{x}}\to\infty,$  when $x\to\infty. $
 \end{definition}

 \begin{definition}{\rm  \cite{kra}} Let $\varphi(x),x\in {\mathbb R},$ be an Orlicz N-function. The
function $\varphi^{*}(x):=\sup_{y\in {\mathbb R}}(xy-\varphi(y)),$ $x\in {\mathbb R},$ is called the
Young-Fenchel transform (also known as the Legendre transform) of $\varphi(\cdot).$  \end{definition}
The function $\varphi^{*}(\cdot)$ is also an Orlicz N-function.

 \begin{definition} {\rm \cite{kra}} An Orlicz N-function  $\varphi(\cdot)$ satisfies {\bf Condition Q}  if
\[\lim\limits_{x\rightarrow 0}{\varphi(x)}/{x^2}=C>0,\]
where the constant $C$ can be equal to $+\infty.$ 
 \end{definition}

 \begin{example} The following functions are N-functions that satisfy   Condition Q: \[\varphi(x)=C{|x|^{\alpha}},\ 1<\alpha\le2;\quad \varphi(x)=\exp\{Cx^2\}-1;\]
\[\varphi(x)=\begin{cases} Cx^2, &\mbox{if } |x|\le 1, \\
C|x|^{\alpha}, & \mbox{if } |x|> 1, \end{cases} \quad \alpha>2,\] 
where $C>0.$
 \end{example}

 \begin{lemma}\label{lem1} {\rm \cite{kra}} Let  $\varphi(\cdot)$  be an Orlicz N-function. Then it can be represented as
$\varphi(u)=\int_{0}^{|u|}f(v)\,dv,$ where $f(\cdot)$ is
a monotonically non-decreasing, right-continuous function, such that
$f(0)=0$ and $f(x)\to +\infty,$ when  $x\to +\infty.$   \end{lemma}

Let $\{\Omega, \cal{B}, \mathbf {P}\}$ be a standard probability space and $L_p(\Omega)$ denote a space of random variables having finite $p$-th absolute moments.

 \begin{definition} {\rm \cite{ant, koz4}} Let $\varphi(\cdot)$ be an Orlicz
N-function satisfying the {Condition Q.} A zero mean random
variable $\xi$ belongs to the space $Sub_\varphi(\Omega)$ (the space
of $\varphi$-sub-Gaussian random variables), if there exists a
constant $a_{\xi}\geq 0$ such that the inequality $\mathbf
E\exp\left(\lambda\xi\right)\le
\exp\left(\varphi(a_{\xi}\lambda)\right)$ holds for all $\lambda\in
{\mathbb R}.$  \end{definition}
The space $Sub_\varphi(\Omega)$ is a Banach space with respect to
the norm (see  \cite{bul}) 
$$
\tau_\varphi \left( \xi \right): = \mathop {\sup }\limits_{\lambda
\ne 0} \frac{\varphi ^{\left( { - 1} \right)}\left( {\ln \mathbf
E\exp
\left\{ {\lambda \xi } \right\}} \right)}{\left| \lambda
\right|},
$$
where $\varphi ^{\left( { - 1} \right)}(\cdot)$ denotes the inverse function of $\varphi (\cdot).$

If  $\varphi(x)={x^2}/2$ then $Sub_\varphi(\Omega)$ is called a space of subgaussian random variables. This space was introduced in the article  \cite{kah}.

 \begin{definition} {\rm \cite{bul}} Let $\mathbf{T}$ be a parametric space. A random process $\mathbf
X(t),$ $t \in \mathbf{T},$ belongs to the space $Sub_{\varphi}(\Omega)$ if $\mathbf X(t)\in Sub_{\varphi}(\Omega)$ for
all $t\in \mathbf{T}.$ 
 \end{definition}
A Gaussian centered random process $\mathbf X(t),$ $t\in \mathbf{T},$ belongs to the space $Sub_{\varphi}(\Omega),$  where $\varphi(x)={x^2}/2$ and  $\tau_{\varphi}(\mathbf X(t)) =\left(\mathbf
E\left|\mathbf X(t)\right|^2\right)^{1/2}.$

 \begin{definition}\label{SSub} {\rm \cite{koz1}} A family $\Xi$ of random
variables $\xi\in Sub_{\varphi}(\Omega)$ is called strictly
$Sub_{\varphi}(\Omega)$ if there exists a constant $C_{\Xi}>0$
such that for any finite set $I,$ $\xi_{i}\in \Xi,$ $i\in I,$
and for arbitrary $\lambda_{i}\in {\mathbb R},$ $i\in I:$ 
\[\tau_{\varphi}\left(\sum\limits_{i \in
I}\lambda_{i}\xi_{i}\right)\leq
 C_{\Xi}\left(\mathbf E\left(\sum\limits_{i \in I}\lambda_{i}\xi_{i}\right)^2\right)^{1/2}.\]
  \end{definition}
$C_{\Xi}$ is called a determinative constant.
The strictly $Sub_{\varphi}(\Omega)$ family will be denoted by $SSub_{\varphi}(\Omega).$ 

 \begin{definition} {\rm \cite{koz1}} A $\varphi$-sub-Gaussian random
process $\mathbf X(t),$ $t\in \mathbf{T},$ is called strictly
$Sub_{\varphi}(\Omega)$ if the family of random variables $\{\mathbf
X(t),t\in \mathbf{T}\}$ is strictly $Sub_{\varphi}(\Omega).$ The determinative
constant of this family is called a determinative constant of the
process and denoted by $C_{\mathbf X}$.
 \end{definition}

A Gaussian centered random process $\mathbf X(t),$ $t\in \mathbf{T},$ is a 
$SSub_{\varphi}(\Omega)$ process, where $\varphi(x)={x^2}/2$ and the
determinative constant $C_{\mathbf X}=1.$

\section{Approximation in $L_p([0,T])$}\label{sec4}
This section presents results on the WKS approximation of $Sub_{\varphi}(\Omega)$ and $SSub_{\varphi}(\Omega)$ random processes in  $L_p([0,T])$ with a given accuracy and reliability. Various specifications of the general results are obtained for important scenarios. Notice, that the approximation in $L_p([0,T])$ investigates the closeness of trajectories of $\mathbf X(t)$ and $\mathbf X_n(t),$ see, e.g.,  \cite{koz0, kozol1, kozol2}. It is different from the known $L_p$-norm results which give the closeness of $\mathbf X(t)$ and $\mathbf X_n(t)$ distributions for each $t,$ see, e.g.,  \cite{he, ole1, ole0}. 

First, we state some auxiliary results that we need for Theorems~\ref{the3} and~\ref{the4}. 

Let $\{\mathbf{T},\mathfrak{S},\mu\}$ be a measurable space and $\mathbf X(t),$ $t\in \mathbf{T},$
be a random process from the space $Sub_{\varphi}(\Omega).$
We will use the following notation $\tau_{\varphi}(t):=\tau_{\varphi}\left(\mathbf X(t)\right)$ for the norm of $\mathbf X(t)$ in the space $Sub_{\varphi}(\Omega).$

There are some general results in the literature which can be used to obtain asymptotics of the tail of power functionals of sub-Gaussian processes, see, for example, \cite{bor}. In contrast to these asymptotic results, numerical sampling applications require non-asymptotic bounds with an explicit range over which they can be used. The following theorem provides such bounds for the case of $\varphi$-sub-Gaussian processes.

 \begin{theorem}\label{the1}{\rm \cite{koz0}}  Let $p\ge 1$ and  
\[c:=\int_\mathbf{T}
\left(\tau_{\varphi}(t)\right)^p\,d\mu(t)<\infty.\]  Then the integral $\int_\mathbf{T}
\left|\mathbf X(t)\right|^p\,d\mu(t)$ exists with probability {\rm 1} and the following
inequality holds

\begin{equation}\label{est}\mathbf P\left\{\int_\mathbf{T}
\left|\mathbf
X(t)\right|^p\,d\mu(t)>\varepsilon\right\}\le 2\exp\left\{-\varphi^*\left(\left({\varepsilon}/c\right)^{1/p}\right)\right\}
\end{equation}
for each non-negative
\begin{equation}\label{epsilon_ner}\varepsilon>c\cdot\left(f\left(p(c/\varepsilon)^{1/p}\right)\right)^p,\end{equation}
where $f(\cdot)$ is a density of $\varphi(\cdot)$ defined in  Lemma~{\rm\ref{lem1}.}
 \end{theorem}

 \begin{example} Let $\varphi(x)={|x|^{\alpha}}/\alpha,$ $
1<\alpha\le2.$ Then $f(x)=x^{\alpha-1}$ and
$\varphi^*(x)={|x|^{\gamma}}/\gamma,$ where $\gamma \ge 2$ and $1/{\alpha}+1/{\gamma}=1.$ Hence,
 inequality~{\rm(\ref{epsilon_ner})} can be rewritten as
\[\varepsilon>c\cdot\left(f\left(p(c/\varepsilon)^{1/p}\right)\right)^p=c^{\alpha}p^{(\alpha-1)p}\varepsilon^{1-\alpha}.\]
Therefore, it holds 
\begin{equation}\label{intX(t)gauss_ner}\mathbf P\left\{\int_\mathbf{T}
\left|\mathbf
X(t)\right|^p\,d\mu(t)>\varepsilon\right\}\le2\exp\left\{-\frac1\gamma\left(\frac{\varepsilon}{
c}\right)^{\gamma/p}\right\},
\end{equation}
when
$\varepsilon>c\cdot p^{\frac{\alpha-1}{\alpha}p}.$
\end{example}

\begin{example} If $\mathbf X(t),$ $t\in \mathbf{T},$ is a Gaussian centered random
process, then the inequality
\begin{equation}\label{est12}\mathbf P\left\{\int_\mathbf{T}
\left|\mathbf
X(t)\right|^p\,d\mu(t)>\varepsilon\right\}\le2\exp\left\{-\frac12\left(\frac{\varepsilon}{\tilde
c}\right)^{2/p}\right\}
\end{equation}
holds
true for $\varepsilon>\hat c\cdot p^{\frac{p}{2}},$ where $\hat
c:=\int_\mathbf{T}\left(\mathbf E\left(\mathbf
X(t)\right)^2\right)^{p/2}\,d\mu(t).$ 
 \end{example}
 
 \begin{example} Let $\mathbf X(t)$ be a centered bounded random variable for all $t\in \mathbf{T}.$ Then the process  $\mathbf X(t),$ $t\in \mathbf{T},$ belongs to all spaces $Sub_{\varphi}(\Omega)$ and satisfies  (\ref{est}), (\ref{intX(t)gauss_ner}), and (\ref{est12}). 
  \end{example}
  
   \begin{example}\label{example5} Let $\alpha \ge 2$ and
   \[\varphi(x)=\begin{cases} x^2/\alpha, &\mbox{if } |x|\le 1, \\
   |x|^{\alpha}/\alpha, & \mbox{if } |x|> 1. \end{cases}  \] 
   Then  $\varphi(x)$ is an Orlicz N-function satisfying the {Condition Q.}
   
Let,  for each $t\in \mathbf{T},$ $\mathbf X(t)$  be a two-sided Weibull random variable,  i.e.

\[\mathbf P\left\{\mathbf X(t)\ge x \right\}=\mathbf P\left\{\mathbf X(t) \le -x \right\}=\frac{1}{2}\exp\left\{-\frac{x^\alpha}{\alpha}\right\}, \quad x>0.\]  

Then  $\mathbf X(t),$ $t\in \mathbf{T},$ is a random process from the space $Sub_{\varphi}(\Omega)$ and Theorem~\ref{the1} holds true for
\[f(v)=\begin{cases} 2v/\alpha, &\mbox{if } |v| < 1, \\
   |v|^{\alpha-1}, & \mbox{if } |v|\ge 1, \end{cases} \quad \mbox{and}\quad \varphi^*(x)=\begin{cases} \alpha x^2/4, &\mbox{if } 0\le |x|\le {2}/{\alpha}, \\
      |x|-{1}/{\alpha}, &\mbox{if } {2}/{\alpha} < |x|\le 1, \\
      |x|^{\gamma}/\gamma, & \mbox{if } |x|> 1, \end{cases}  \] 
where $\gamma\in (1,2]$ and  $1/\alpha+1/\gamma=1.$  
    \end{example}

 \begin{theorem}\label{the3} Let $\omega>\Lambda>0,$ $n\ge \frac{\omega t}{\pi \sqrt{z}},$ $z\in(0,1).$ Let $\mathbf X(t),$ $t\in \mathbf{R},$ be a stationary
$SSub_{\varphi}(\Omega)$ process which spectrum is bandlimited to $[-\Lambda,\Lambda),$  
 $\mathbf X_n(t)$ be defined by {\rm(\ref{shan2})}, and 
\[S_{n,p}:=\left(\frac{C_{\mathbf X}}{n}\right)^p\int_0^T C_n^{p/2}(t)\,dt,\]
where $C_{\mathbf X}$ is a determinative constant of the process  $\mathbf X(t),$ $C_n(t)$ is given by {\rm(\ref{Cn}).}

Then,  $\int_0^T
\left|\mathbf X(t)-\mathbf X_n(t)\right|^p\,dt$ exists with probability {\rm 1} and  the following inequality holds true for $\varepsilon>S_{n,p}\cdot\left(f\left(p\,(S_{n,p}/\varepsilon)^{1/p}\right)\right)^p:$
\[\mathbf P\left\{\int_0^T
\left|\mathbf
X(t)-\mathbf
X_n(t)\right|^p\,dt>\varepsilon\right\}\le 2\exp\left\{-\varphi^*\left(\left({\varepsilon}/S_{n,p}\right)^{1/p}\right)\right\}.
\]
 \end{theorem}
\begin{proof} It follows from (\ref{shan2}) and Definition~\ref{SSub} that $\mathbf X(t)-\mathbf X_n(t)$ is a $SSub_{\varphi}(\Omega)$ random process with  the determinative constant $C_{\mathbf X}.$ 

Applying  Theorem~\ref{the1} to $\mathbf X(t)-\mathbf X_n(t)$ for the case $\mathbf T=[0,T]$ and the Lebesgue measure $\mu$ on $[0,T]$ we obtain that $\int_0^T
\left|\mathbf X(t)-\mathbf X_n(t)\right|^p\,dt$ exists with probability {\rm 1} and
\[\mathbf P\left\{\int_0^T
\left|\mathbf
X(t)-\mathbf
X_n(t)\right|^p\,dt>\varepsilon\right\}\le 2\exp\left\{-\varphi^*\left(\left({\varepsilon}/c\right)^{1/p}\right)\right\},
\]
where $c:=\int_0^T \left(\tau_{\varphi}(\mathbf
X(t)-\mathbf
X_n(t))\right)^p\,dt.$

Notice that $\varphi^*\left(\cdot\right)$ and $f(\cdot)$ are  monotonically  non-decreasing. Therefore,  for any $\tilde{c}\ge c$ we obtain
\[\tilde{c}\cdot\left(f\left(p(\tilde{c}/\varepsilon)^{1/p}\right)\right)^p\ge c\cdot\left(f\left(p(c/\varepsilon)^{1/p}\right)\right)^p,\]
\[ \exp\left\{-\varphi^*\left(\left({\varepsilon}/c\right)^{1/p}\right)\right\}\le \exp\left\{-\varphi^*\left(\left({\varepsilon}/\tilde{c}\right)^{1/p}\right)\right\}.
\]

Hence, the statement of Theorem~\ref{the1} holds true if the constant $c$  in (\ref{est}) and (\ref{epsilon_ner}) is replaced by some $\tilde{c},$ $\tilde{c}\ge c.$ 
Now, by Definition~\ref{SSub} and part 1 of Theorem~\ref{the0} one can choose  $\tilde{c}=S_{n,p}$  which finishes the proof of the theorem.
\end{proof}  
\begin{example} Recalling that in the Gaussian case $\varphi^*(x)=|x|^2/2$ we obtain the following specification of the above theorem.

If $\mathbf X(t),$ $t\in \mathbf{R},$ is a Gaussian process, then for $\varepsilon>\hat{S}_{n,p}\cdot p^{p/2}$ it holds
\[\mathbf P\left\{\int_0^T
\left|\mathbf
X(t)-\mathbf
X_n(t)\right|^p\,dt>\varepsilon\right\}\le 2\exp\left\{-\frac{1}{2}\left(\frac{\varepsilon}{\hat{S}_{n,p}}\right)^{2/p}\right\},
\]
where \[\hat{S}_{n,p}:=n^{-p}\int_0^T C_n^{p/2}(t)\,dt.\]
\end{example}

\begin{example} Let $\mathbf B(\tau)$ be a covariance function that corresponds to a bandlimited spectrum and has the following Mercer's representation 
\[\mathbf B(t-s)=\mathbf E\mathbf X(t) \mathbf X(s) = \sum_{j=1}^\infty \lambda_j \, e_j(s) \, e_j(t),\quad t,s\in \mathbf{R},\]
where $\lambda_j$ and $e_j(s)$ are eigenvalues and eigenfunctions,  respectively, associated to $\mathbf B(t,s).$ 

Let us define the corresponding stochastic process $\mathbf X(t),$ $t\in \mathbf{R},$ using the Karhunen-Lo\'{e}ve type expansion 
\[\mathbf X(t)= \sum_{j=1}^\infty \xi_j  e_j(t),\]
where $\xi_j, j\ge 1,$ are independent identically distributed random variables from the space $Sub_{\varphi}(\Omega).$ If $\varphi(\sqrt{x})$ is a convex function, then  $\mathbf X(t),$ $t\in \mathbf{R},$ is a $SSub_{\varphi}(\Omega)$ stochastic process, see \cite{koz1}.

For example, let $\xi_j, j\ge 1,$ be two-sided Weibull random variables defined in Example~\ref{example5}. Then Theorem~\ref{the3} holds true provided that the functions $f(v)$ and $\varphi^*(x)$ are selected as in Example~\ref{example5}. 
\end{example}

 \begin{definition}\label{def8} We say that  $\mathbf X_n$ approximates $\mathbf X$  in $L_p([0,T])$ with accuracy $\varepsilon>0$ and reliability $1-\delta,$ $0<\delta < 1,$ if 
\[\mathbf P\left\{\int_0^T \left|\mathbf X(t)-\mathbf X_n(t)\right|^p\,dt >\varepsilon\right\}\le \delta.\]
 \end{definition}

Using Definition~\ref{def8}  and Theorem~\ref{the3} we get the following result.
 \begin{theorem}\label{the4}  Let $\mathbf X(t),$ $t\in \mathbf{R},$ be a stationary
$SSub_{\varphi}(\Omega)$ process with a bounded spectrum. Then $\mathbf X_n$  approximates $\mathbf X$  in $L_p([0,T])$ with accuracy $\varepsilon$ and reliability $1-\delta$ if the following inequalities hold true
\[\label{eps1}\varepsilon>S_{n,p}\cdot\left(f\left(p\,(S_{n,p}/\varepsilon)^{1/p}\right)\right)^p, 
\]
\[
 \exp\left\{-\varphi^*\left(\left({\varepsilon}/S_{n,p}\right)^{1/p}\right)\right\}\le \delta/2.
\]
 \end{theorem}

\begin{cor}\label{cor2}  If $\mathbf X(t),$ $t\in \mathbf{R},$ is a Gaussian process, $\mathbf X_n$  approximates $\mathbf X$  in $L_p([0,T])$ with accuracy $\varepsilon$ and reliability $1-\delta$ if
\begin{equation}\label{Snp}\hat{S}_{n,p}<\frac{\varepsilon}{\max\left(p^{p/2},\left(2\ln(2/\delta)\right)^{p/2}\right)}.\end{equation}
\end{cor}

The next example illustrates an application of the above results for determining the number of terms in the WKS expansions to ensure the approximation of $\varphi$-sub-Gaussian processes with given accuracy and reliability.
 \begin{example}\label{ex2}
Let $p\ge 1$ in Corollary~{\rm\ref{cor2}}. Then by part~{\rm 1}  of Theorem~{\rm\ref{the0}}, for arbitrary $z\in (0,1)$ and $n \ge \frac{\omega T}{\pi \sqrt{z}},$ we get the following estimate
\[\hat{S}_{n,p}\le \left(\frac{\sqrt{\mathbf B(0)}}{n}\right)^{p}\int_0^T \left(A_1t+A_0\right)^pdt\le \frac{\left(\mathbf B(0)\right)^{p/2}T \left(A_1T+A_0\right)^p}{n^p}.\]
where $A_1:=\frac{4\omega}{\pi^2(1-z)}$ and $A_0:=\frac{4(z+2)}{\pi(1-z)^2(1-\Lambda/\omega)}.$

Hence, to guarantee {\rm(\ref{Snp})} for given $p,$ $\varepsilon$ and $\delta$ it is enough to choose an  $n$ such that the following inequality holds true
\[\frac{\left(\mathbf B(0)\right)^{p/2}T \left(A_1T+A_0\right)^p}{n^p}\le \frac{\varepsilon}{\max\left(p^{p/2},\left(2\ln(2/\delta)\right)^{p/2}\right)}\]
for $z=\frac{\omega^2 T^2}{\pi^2n^2}<1.$

For example, for $p=2,$ $T=B(0)=\omega=1,$ and $\Lambda=3/4$ the number of terms $n$ as a function of  $\varepsilon$ and $\delta$ is shown in Figure~{\rm\ref{fig:1}}. It is clear that $n$ increases when $\varepsilon$ and $\delta$ approach~0. However, for reasonably small $\varepsilon$ and $\delta$ we do not need too many sampled values.
\begin{figure}[htb]
\begin{center}
\includegraphics[width=8cm,height=8cm,trim=2mm 1mm 2mm 2mm,clip]{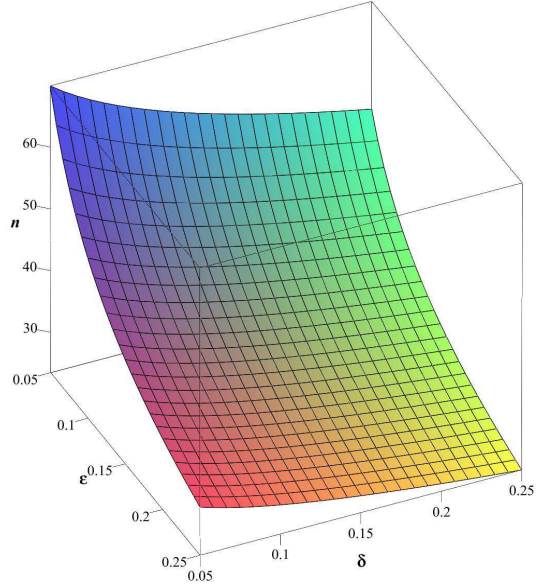}
\caption{The number of terms to ensure specified accuracy and reliability}
\label{fig:1}
\end{center}
\end{figure}

Now, for fixed $\varepsilon$ and $\delta$ Figure~{\rm\ref{fig:2}} illustrates the behaviour of the number of terms $n$ as a function of the parameter~$p\in[1,2].$  The plot was produced using the values $T=B(0)=\omega=1,$ $\Lambda=3/4,$ and $\varepsilon=\delta=0.1.$   
\begin{figure}[htb]
\begin{center}
\includegraphics[width=9cm,height=7cm,trim=2mm 1mm 2mm 2mm,clip]{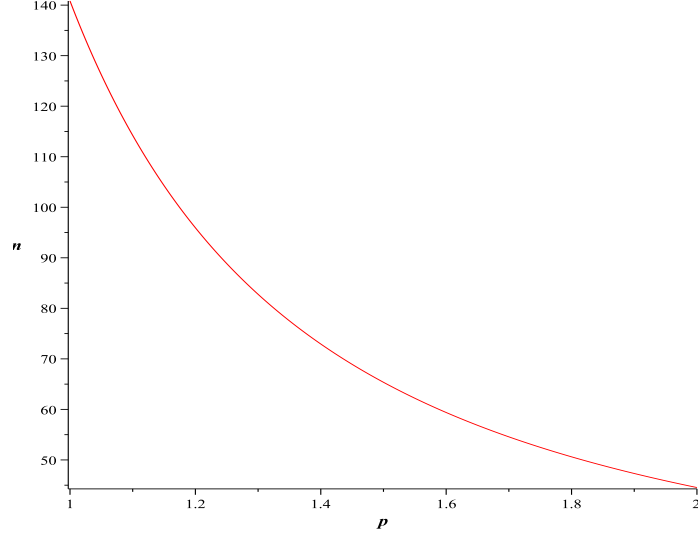}
\caption{The number of terms as a function of $p$}
\label{fig:2}
\end{center}
\end{figure}
 \end{example}

\section{Uniform approximation}

Most of stochastic sampling results commonly seen in the literature concern the mean-square convergence, but various practical applications require uniform convergence. To give an adequate description of the performance of sampling approximations in both cases, for points where the processes are relatively smooth and points where spikes occur, one can use the uniform distance instead of the mean-square one. 
The development of uniform stochastic approximation methods is one of
frontiers in applications of stochastic sampling theory to  modern functional data analysis.

In this section we present results on uniform truncation error upper
bounds appearing in the approximation $\mathbf X(t)\approx \mathbf X_n(t)$  of $Sub_{\varphi}(\Omega)$ and $SSub_{\varphi}(\Omega)$ random processes. We also give some specifications of the general results for which the assumptions can be easily verified.  

Let $\mathbf X(t),$ $t\in \mathbf{T},$ be a $\varphi$-subgaussian random process. It generates the pseudometrics $\rho_{\mathbf X}(t,s)=\tau_{\varphi}(\mathbf X(t)-\mathbf X(s))$  on $\mathbf{T}.$ Let the pseudometric space  $(\mathbf{T},\rho_{\mathbf X})$ be separable, $\mathbf X$ be a separable process, and $\varepsilon_0:=\sup_{t\in \mathbf{T}} \tau_{\varphi}(t)<+\infty.$ 

\begin{definition}{\rm \cite{bul}} Let $N(v)$ denote the smallest number of elements in an $v$-covering of $\mathbf{T},$  i.e. the smallest number of closed balls $B_i,$ $i\in I,$ of diameters at most $2v$ and such that $\cup_{i \in I}B_i=\mathbf{T}.$ The function  $N(v),$ $v>0,$ is called the metric massiveness of the space $\mathbf{T}$ with respect to the pseudometric $\rho_{\mathbf X}.$ The function $H(v):=\ln N(v)$  is called the metric entropy of the space $\mathbf{T}$ with respect to the pseudometric $\rho_{\mathbf X}.$ 
\end{definition}

 Note that the function  $N(v)$ coincides with the number of point in a minimal $v$ net covering the space $\mathbf{T}$ and can be equal $+\infty.$ 
 
 Entropy methods to study the metric massiveness of function calsses and spaces play an important role in modern approximation theory. Various properties and numerous examples of the metric massiveness and the metric entropy can be found in \cite[\S 3.2]{bul}.

 \begin{theorem}\label{the5}{\rm\cite{bul}}  Let $r(x),$ $x\ge 1,$ be a non-negative, monotone increasing function such that the function $r\left(e^x\right),$  $x\ge 1,$ is convex and 
\[I_r(v):=\int_0^vr(N(v))dv<+\infty,\]
where $N(v)$ is the massiveness of the pseudometric space $(\mathbf{T},\rho_{\mathbf X}).$

Then, for all $\lambda >0,$ $0<\theta <1,$ it holds
\begin{equation}\label{Q}\mathbf E \exp\left\{\lambda \sup_{t\in \mathbf T}|\mathbf X(t)|\right\}\le 2 Q(\lambda,\theta)
\end{equation}
and 
\begin{equation}\label{A}\mathbf P\left\{\sup_{t\in \mathbf T}|\mathbf X(t)|\ge u\right\}\le 2 A(\theta,u),
\end{equation}
where
\[Q(\lambda,\theta):= \exp\left\{\varphi\left(\frac{\lambda \varepsilon_0}{1-\theta}\right)\right\}\, r^{(-1)}\left(\frac{I_r(\theta \varepsilon_0)}{\theta \varepsilon_0}\right),\]
\[A(\theta,u):= \exp\left\{-\varphi^*\left(\frac{u(1-\theta)}{ \varepsilon_0}\right)\right\}\, r^{(-1)}\left(\frac{I_r(\theta \varepsilon_0)}{\theta \varepsilon_0}\right).\]
 \end{theorem}
Below we give a proof of Theorem~\ref{the5} which corrects the version  with mistakes and the missing proof which appeared in  \cite[page~107]{bul}. 
\begin{proof}
We will use  the following inequality from  \cite[page~103]{bul}
\[ \mathbf E \exp\left\{\lambda \sup_{t\in \mathbf T}|\mathbf X(t)|\right\}\le \prod_{k=1}^{\infty}\left[2N(\theta^k\varepsilon_0)\cdot  \exp\left\{\varphi\left(\lambda q_k \theta^{k-1}\varepsilon_0\right)\right\} \right]^{1/q_k},\]
where $(q_k)_{k=1}^\infty$ is a sequence satisfying the inequality $\sum_{k=1}^\infty q_k^{-1}\le 1.$

It is easily seen that
\[\label{qk1} \mathbf E \exp\left\{\lambda \sup_{t\in \mathbf T}|\mathbf X(t)|\right\}\le 2^{\sum_{k=1}^{\infty}q_k^{-1}}\exp \left\{ \sum_{k=1}^{\infty}  \frac{H(\theta^k\varepsilon_0)}{q_k}+ \frac{\varphi\left(\lambda q_k \theta^{k-1}\varepsilon_0\right)}{q_k}\right\} ,\]
where $H(v)$ is the metric entropy of the pseudometric space $(\mathbf{T},\rho_{\mathbf X}).$

Let $q_k={1}/{\theta^{k-1}(1-\theta)}.$ Then it follows from the convexity of $r(e^x)$ that
\[ \mathbf E \exp\left\{\lambda \sup_{t\in \mathbf T}|\mathbf X(t)|\right\}\le 2\exp \left\{ \sum_{k=1}^{\infty} \theta^{k-1}(1-\theta) H(\theta^k\varepsilon_0)+ \varphi\left(\frac{\lambda\varepsilon_0}{1-\theta}\right)\sum_{k=1}^{\infty} \theta^{k-1}(1-\theta)\right\}\]
\[=2 \exp \left\{ \varphi\left(\frac{\lambda\varepsilon_0}{1-\theta}\right)\right\}\cdot r^{(-1)}\left(r\left(\exp \left\{ \sum_{k=1}^{\infty} \theta^{k-1}(1-\theta) H(\theta^k\varepsilon_0)\right\}\right)\right) \]
\[\le 2 \exp \left\{ \varphi\left(\frac{\lambda\varepsilon_0}{1-\theta}\right)\right\}\cdot r^{(-1)}\left( \sum_{k=1}^{\infty} \theta^{k-1}(1-\theta)\cdot r\left(N(\theta^k\varepsilon_0)\right)\right). \]
From the estimate
\[r\left(N(\theta^k\varepsilon_0)\right)\le \frac{1}{\varepsilon_0\theta^{k}(1-\theta)} \int_{\theta^{k+1}\varepsilon_0}^{\theta^{k}\varepsilon_0}r(N(v))dv\]
we deduce that
\[ \sum_{k=1}^{\infty} \theta^{k-1}(1-\theta)\cdot r\left(N(\theta^k\varepsilon_0)\right)\le \frac{1}{\theta\varepsilon_0} \int_{0}^{\theta\varepsilon_0}r(N(v))dv.\]
The above estimates imply the inequality (\ref{Q}).

To prove the inequality (\ref{A}) we note that by (\ref{Q}) for all $\lambda>0$
\begin{eqnarray}\mathbf P\left\{\sup_{t\in \mathbf T}|\mathbf X(t)|\ge u\right\} &\le& \frac{\mathbf E \exp\left\{\lambda\sup_{t\in \mathbf T}|\mathbf X(t)|\right\}}{\exp(\lambda u)}\le 2\,r^{(-1)}\left(\frac{I_r(\theta \varepsilon_0)}{\theta \varepsilon_0}\right)
\nonumber\\
&\times&\exp\left\{-\lambda u+ \varphi\left(\frac{\lambda \varepsilon_0}{1-\theta}\right)\right\}. \nonumber
\end{eqnarray}
By the definition of the Young-Fenchel transform we get
\[\inf_{\lambda \ge 0}\left(-\lambda u+ \varphi\left(\frac{\lambda \varepsilon_0}{1-\theta}\right)\right)=-\sup_{\lambda \ge 0}\left(\frac{\lambda\varepsilon_0}{1-\theta}\cdot\frac{u(1-\theta)}{\varepsilon_0}- \varphi\left(\frac{\lambda \varepsilon_0}{1-\theta}\right)\right)=-\varphi^*\left(\frac{u(1-\theta)}{ \varepsilon_0}\right).\]
This proves the inequality (\ref{A}). 
\end{proof}

 \begin{theorem}\label{the6} Let $\mathbf X(t),$ $t\in [0,T],$ be a separable $\varphi$-subgaussian random process such that $\sup_{t\in [0,T]} \tau_{\varphi}(t)<+\infty$ and 
\begin{equation}\label{S2} \sup_{|t-s|\le h}\tau_\varphi(\mathbf X(t)-\mathbf X(s))\le \sigma(h),
\end{equation} 
where $\sigma(h),$ $h\ge 0,$ is a monotone increasing continuous function such that $\sigma(0)=0.$ Let $r(\cdot)$ be the function introduced in Theorem~{\rm\ref{the5}}.

If 
\[\tilde{I}_r(v):=\int_0^vr\left(\frac{T}{2\sigma^{(-1)}(u)}+1\right)du<+\infty,\]
then for any $\theta\in(0,1)$ and $\varepsilon>0$
\[\mathbf P\left\{\sup_{t\in [0,T]}|\mathbf X(t)|\ge \varepsilon\right\}\le 2 \tilde{A}(\theta,\varepsilon),\]
where 
\[\tilde{A}(\theta,\varepsilon):= \exp\left\{-\varphi^*\left(\frac{\varepsilon(1-\theta)}{ \varepsilon_0}\right)\right\}\, r^{(-1)}\left(\frac{\tilde{I}_r(\theta \varepsilon_0)}{\theta \varepsilon_0}\right).\]
 \end{theorem}
\begin{proof} Notice that the space $([0,T],\rho_{\mathbf X}(t,s))$ is separable. Also,  the next inequality holds true
\[N(u)\le \frac{T}{2\sigma^{(-1)}(u)}+1.\]
Hence, the statement of the theorem follows from Theorem~\ref{the5}.
\end{proof}

\begin{rem} In \cite{vie} Malliavin derivatives were applied to derive some upper bounds  similar to the results in Theorems~\ref{the5} and~\ref{the6}.  However, these bound can not be directly compared with the results in Theorems~\ref{the5} and~\ref{the6} as they are valid only for a range of values of $\varepsilon$ which is separated from 0. 
\end{rem}
Let $\alpha, \gamma\in (1,\infty)$ satisfy $1/\alpha+1/\gamma=1.$
 \begin{example}\label{ex3} Let $\varphi(x)=|x|^\alpha/\alpha,\ 1<\alpha\le2.$ Then
\[\tilde{A}(\theta,\varepsilon)= \exp\left\{-\frac{\varepsilon^\gamma(1-\theta)^\gamma}{\gamma\, \varepsilon_0^\gamma}\right\}\, r^{(-1)}\left(\frac{\tilde{I}_r(\theta \varepsilon_0)}{\theta \varepsilon_0}\right).\]
 \end{example}
 \begin{example}\label{ex4} Let $\sigma(h)=Ch^\kappa,$ $0<\kappa\le 1,$ and $r(v)=(v-1)^\beta,$ $0<\beta<\kappa.$ Then $\sigma^{(-1)}(u)=(u/C)^{1/\kappa},$ $r^{(-1)}(v)=v^{1/\beta}+1,$ and 
\[\tilde{I}_r(v)=\int_0^v \left(\frac{C^{1/\kappa}T}{2u^{1/\kappa}}\right)^{\beta}du =\left(\frac{C^{1/\kappa}T}{2}\right)^{\beta}\left(1-\frac{\beta}{\kappa}\right)^{-1}v^{1-\beta/\kappa}.\]
Hence,
\[r^{(-1)}\left(\frac{\tilde{I}_r(\theta \varepsilon_0)}{\theta \varepsilon_0}\right)=
\frac{C^{1/\kappa}T}{2}\left(1-\frac{\beta}{\kappa}\right)^{-1/\beta}\left(\theta \varepsilon_0\right)^{-1/\kappa}+1\]
and
\[\tilde{A}(\theta,\varepsilon)= \exp\left\{-\varphi^*\left(\frac{\varepsilon(1-\theta)}{ \varepsilon_0}\right)\right\}\, \left(\frac{C^{1/\kappa}T}{2}\left(1-\frac{\beta}{\kappa}\right)^{-1/\beta}\left(\theta \varepsilon_0\right)^{-1/\kappa}+1\right).\]
If $\beta \to 0,$ then $\left(1-\frac{\beta}{\kappa}\right)^{1/\beta}\to e^{-1/\kappa}$ and we obtain the inequality 
\begin{equation}\label{tildeA}\tilde{A}(\theta,\varepsilon) \le \exp\left\{-\varphi^*\left(\frac{\varepsilon(1-\theta)}{ \varepsilon_0}\right)\right\}\, \left(\frac{T}{2}\left(\frac{e\,C}{\theta \varepsilon_0}\right)^{1/\kappa}+1\right).\end{equation}
 \end{example}
\begin{rem} Note that the particular form of $\sigma(h)$ in Example~\ref{ex4} guarantees H\"{o}lder continuity of sample paths of the stochastic process $\mathbf X.$  However, H\"{o}lder exponents  may be different for different functions $\varphi .$
\end{rem}

 \begin{example}\label{ex5} Let $\varphi(x)=|x|^\alpha/\alpha,\ 1<\alpha\le2,$ $\sigma(h)=Ch^\kappa,$ $0<\kappa\le 1,$ and $r(v)=(v-1)^\beta,$ $0<\beta<\kappa.$  Then, by Examples~{\rm\ref{ex3}} and~{\rm\ref{ex4}} it follows that
\[\tilde{A}(\theta,\varepsilon) \le \exp\left\{-\frac{\varepsilon^\gamma(1-\theta)^\gamma}{ \gamma\,\varepsilon_0^\gamma}\right\}\, \left(\frac{T}{2}\left(\frac{e\,C}{\theta \varepsilon_0}\right)^{1/\kappa}+1\right).\]
Let now $\theta={\varepsilon_0}/{\varepsilon}.$ Then for $\varepsilon>\varepsilon_0$  we obtain  $\theta<1$ and
\[\mathbf P\left\{\sup_{t\in [0,T]}|\mathbf X(t)|\ge \varepsilon\right\}\le 2 \exp\left\{-\frac{1}{\gamma}\left(\frac{\varepsilon}{\varepsilon_0}-1\right)^\gamma \right\}\, \left(\frac{T}{2}\left(\frac{e\,\varepsilon\,C}{\varepsilon_0^2}\right)^{1/\kappa}+1\right).\]
 \end{example}

 \begin{theorem}\label{the7} Let $\mathbf X(t),$ $t\in [0,T],$ be a separable $SSub_{\varphi}(\Omega)$ random process whose spectrum is bandlimited to $[-\Lambda,\Lambda).$ Let the  truncated restoration sum  $\mathbf X_n(t)$ for the  process  $\mathbf X(t)$ is given by {\rm(\ref{shan2})}. Then, for any $\theta\in(0,1),$ $\varepsilon >0,$ and such values of $n$ that $z^*:=\frac{\omega^2 T^2}{n^2\pi^2}< 1:$ 
\[\mathbf P\left\{\sup_{t\in [0,T]}|\mathbf X(t)-\mathbf X_n(t)|\ge \varepsilon\right\}\le  \exp\left\{-\varphi^*\left(\frac{\varepsilon(1-\theta)}{ C_n}\right)\right\}\, \left(\frac{eTC_{\mathbf X}\sqrt{b_n}}{2n\theta C_n}+1\right),\]
where 
$C_{\mathbf X}$ is the determinative constant of the process $\mathbf X(t)$, $b_n:=b_n(T,T)$ is given by~{\rm(\ref{bn})} evaluated at $z=z^*,$
\[C_n:= \frac{C_{\mathbf X}\mathbf B(0)}{n} \cdot\left(\frac{4\omega T}{\pi^2 (1-z^*)}+\frac{4\left(z^*+1+\frac{1}{n}\right)}{\pi (1-z^*)^2\left(1-\frac{\Lambda}{\omega}\right)} \right)^2.\] 
 \end{theorem}
 \begin{proof}  It follows from (\ref{shan2}) and Definition~\ref{SSub} that $\mathbf Y_n(t)=\mathbf X(t)-\mathbf X_n(t)$ is a $SSub_{\varphi}(\Omega)$ random process with  the determinative constant $C_{\mathbf X}.$ Hence, by Definition~\ref{SSub} and an application of part 1 of  Theorem~\ref{the0} to $\mathbf Y_n(t)$ we get  
 \[\tilde{\varepsilon}_0=\sup_{t\in [0,T]}\tau_{\varphi}\left(\mathbf Y_n(t)\right)\le C_{\mathbf X} \sup_{t\in [0,T]} \left(\mathbf E\mathbf Y_n^2(t)\right)^{1/2}\le \frac{C_{\mathbf X}}{n}\sup_{t\in [0,T]}\sqrt{C_n(t)}.\]
 
 Notice, that it follows from $n \ge \frac{\omega t}{\pi \sqrt{z}}$ in part 1 of Theorem~\ref{the0} and (\ref{Cn}) that $C_n(t)$ is an increasing function of $T$ and $z.$ Therefore, \[\sup_{t\in [0,T]}{C_n(t)}=\mathbf B(0)\cdot\sup_{0<z\le z*}  \left(\frac{4\omega T}{\pi^2 (1-z)}+\frac{4\left(z+1+\frac{1}{n}\right)}{\pi (1-z)^2\left(1-\frac{\Lambda}{\omega}\right)} \right)^2=\frac{n\,C_n}{C_{\mathbf X}}\]
 and
$\tilde{\varepsilon}_0 \le C_n.$  

By Definition~\ref{SSub} and an application of part 2 of  Theorem~\ref{the0} to $\mathbf Y_n(t)$ we get
\[ \sup_{|t-s|\le h}\tau_\varphi(\mathbf Y_n(t)-\mathbf Y_n(s))\le C_{\mathbf X} \sup_{|t-s|\le h}  \left(\mathbf E\left|\mathbf Y_n(t)-\mathbf Y_n(s)\right|^2\right)^{1/2}\]
\[\le C_{\mathbf X}\sup_{|t-s|\le h}
\frac{|t-s|}{n} \sqrt{b_n(t,s)}.
\] 
It follows from (\ref{bn}) that $b_n(t,s)$ is an increasing function of its arguments $t,s,$ and parameter~$z.$ Hence, $\sup_{|t-s|\le h}
b_n(t,s)\le b_n(T,T)\le b_n$ and the condition~(\ref{S2}) of Theorem~\ref{the6} is satisfied for the function $\sigma(h)=C_{\mathbf X}\sqrt{b_n}\cdot h/n.$ Therefore, we can apply the result (\ref{tildeA}) where $\kappa=1$ and $C=C_{\mathbf X}\sqrt{b_n}/n.$

Analogously to the proof of Theorem~\ref{the3} one can show that the upper bound remains valid if the constant $\varepsilon_0$ in the expression $\tilde{A}(\theta,\varepsilon)$ is replaced by a larger value. Hence,  an application of Theorem~\ref{the6} to $\mathbf Y_n(t)$ and the above estimates give

\[\tilde{A}(\theta,\varepsilon)\le  \exp\left\{-\varphi^*\left(\frac{\varepsilon(1-\theta)}{ C_n}\right)\right\}\, \left(\frac{eTC_{\mathbf X}\sqrt{b_n}}{2n\theta C_n}+1\right) \]
which completes the proof.
 \end{proof}
 \begin{cor}  Let  $\varphi(x)=|x|^\alpha/\alpha,\ 1<\alpha\le2,$ in Theorem~{\rm\ref{the7}}. Then, by Example~{\rm\ref{ex5}} for $\varepsilon > C_n$  it holds
 \[\mathbf P\left\{\sup_{t\in [0,T]}|\mathbf X(t)-\mathbf X_n(t)|\ge \varepsilon\right\}\le 2 \exp\left\{-\frac{1}{\gamma}\left(\frac{\varepsilon}{C_n}-1\right)^\gamma \right\}\, \, \left(\frac{\varepsilon eTC_{\mathbf X}\sqrt{b_n}}{2n C^2_n}+1\right).\]
  \end{cor}
It follows from the definition of $C_n$  that $C_n\sim 1/n,$ when $n\to \infty.$ Hence, for a fixed value of $\varepsilon$ the right-hand side of the above inequality vanishes when $n$ increases.

Similarly to Section~\ref{sec4} one can define the uniform approximation of $\mathbf X(t)$ with a given accuracy  and reliability.

\begin{definition}\label{def9}  $\mathbf X_n(t)$ uniformly approximates $\mathbf X(t)$   with accuracy $\varepsilon>0$ and reliability $1-\delta,$ $0<\delta < 1,$ if 
\[\mathbf P\left\{\sup_{t\in [0,T]}|\mathbf X(t)-\mathbf X_n(t)| >\varepsilon\right\}\le \delta.\]
 \end{definition}

By Definition~\ref{def9}  and Theorem~\ref{the7} we obtain the following result.
 \begin{theorem}\label{the8}  Let $\mathbf X(t),$ $t\in \mathbf{R},$ be a separable
$SSub_{\varphi}(\Omega)$ process with a bounded spectrum, $\theta\in(0,1),$ $\varepsilon >0,$ and  $n$ is such an positive integer number that $z^*:=\frac{\omega^2 T^2}{n^2\pi^2}< 1.$  Then,  $\mathbf X_n(t)$  uniformly approximates $\mathbf X(t)$  with accuracy $\varepsilon$ and reliability $1-\delta$ if the following inequality holds true
\[
\exp\left\{-\varphi^*\left(\frac{\varepsilon(1-\theta)}{ C_n}\right)\right\}\, \left(\frac{eTC_{\mathbf X}\sqrt{b_n}}{2n\theta C_n}+1\right)\le \delta.
\]
 \end{theorem}

\begin{cor}\label{cor3}  Let  $\varphi(x)=|x|^\alpha/\alpha,\ 1<\alpha\le2,$ in Theorem~{\rm\ref{the8}}. Then,  $\mathbf X_n(t)$  uniformly approximates $\mathbf X(t)$  with accuracy $\varepsilon$ and reliability $1-\delta$ if $\varepsilon > C_n$ and 
\[ \exp\left\{-\frac{1}{\gamma}\left(\frac{\varepsilon}{C_n}-1\right)^\gamma \right\}\, \, \left(\frac{\varepsilon eTC_{\mathbf X}\sqrt{b_n}}{2n C^2_n}+1\right)< \delta/2.\] 
\end{cor}

Notice that  for Gaussian processes $\mathbf X(t)$ all results of this section hold true when $\alpha=\gamma=2$ and $C_{\mathbf X}=1.$ 

\section{Conclusions}\label{sec6}
These results may have various applications for the approximation  of stochastic processes. The obtained rate of convergence provides a constructive algorithm for determining the number of terms in the WKS expansions to ensure the approximation of $\varphi$-sub-Gaussian processes with given accuracy  and reliability. 
The developed methodology and new estimates are important extensions of the known results in the  stochastic sampling  theory to the space $L_p([0,T])$ and the class of $\varphi$-sub-Gaussian random processes. 
In addition to classical applications of $\varphi$-sub-Gaussian random processes in signal processing, the results can also be used in new areas, for example, compressed sensing and actuarial modelling, see, e.g., \cite{hog, xue, yam}. 

It would be of interest 
\begin{itemize}
\item to apply this methodology to other WKS sampling problems, for example, shifted sampling, irregular  sampling, aliasing errors, see \cite{ole1,ole0,ole2} and references therein; 
\item to derive analogous results for the multidimensional case and  random fields;
\item to derive similar results for the sub-Gaussian case by the generic chaining method and to compare them with the obtained bounds.
\end{itemize}

\section*{Acknowledgements }
This research was partially supported under Australian Research Council's Discovery Projects funding scheme (project number DP160101366) and  La Trobe University DRP Grant in Mathematical and Computing Sciences.  The authors are also grateful for the referee's careful reading of the paper and suggestions, which helped to improve the paper.

\end{document}